\documentclass{amsart}
\usepackage{amsmath, amssymb, amsthm, comment}

\renewcommand{\qed}{$\blacksquare$}
\newcommand{\A}{\mathcal{A}}
\newcommand{\s}{\mathcal{S}}
\newcommand{\U}{\mathcal{U}}
\newcommand{\p}{\mathcal{P}}
\newcommand{\pmra}{\mathcal{P}^{\textrm{MRA}}}
\newcommand{\pso}{\mathcal{P}^{\textrm{SO}}}

\newcommand{\stilde}{\widetilde{S}}

\newcommand{\M}{\mathcal{M}}
\newcommand{\R}{\mathbb{R}}
\newcommand{\N}{\mathbb{N}}
\newcommand{\C}{\mathbb{C}}

\newcommand{\Z}{\mathbb{Z}}
\newcommand{\T}{\mathbb{T}}

\newtheorem{theorem}{Theorem}[section]
\newtheorem{corollary}[theorem]{Corollary}
\newtheorem{lemma}[theorem]{Lemma}
\newtheorem{proposition}[theorem]{Proposition}
\newtheorem{definition}[theorem]{Definition}
\theoremstyle{definition}

\newtheorem{notation}[theorem]{Notation}
\def\proof{\par\medskip\noindent {\sc Proof. }}

\renewcommand{\qed}{\hfill $\blacksquare$}

\begin{document}
\title{Projections and Dyadic Parseval Frame MRA wavelets} 

\author{Peter M. Luthy}
\address{Department of Mathematics, Washington University in St. Louis, MO 63130}
\email{luthy@math.wustl.edu}

\author{Guido L. Weiss}
\address{Department of Mathematics, Washington University in St. Louis, MO 63130}
\email{guido@math.wustl.edu}

\author{Edward N. Wilson}
\address{Department of Mathematics, Washington University in St. Louis, MO 63130}
\email{enwilson@math.wustl.edu}


\begin{abstract}
\noindent A classical theorem attributed to Naimark states that, given a Parseval frame $\mathcal{B}$ in a Hilbert space $\mathcal{H}$, one can embed $\mathcal{H}$ in a larger Hilbert space $\mathcal{K}$ so that the image of $\mathcal{B}$ is the projection of an orthonormal basis for $\mathcal{K}$.  In the present work, we revisit the notion of Parseval frame MRA wavelets from \cite{PSWXI} and \cite{PSWXII} and produce an analog of Naimark's theorem for these wavelets at the level of their scaling functions.  We aim to make this discussion as self-contained as possible and provide a different point of view on Parseval frame MRA wavelets than that of \cite{PSWXI} and \cite{PSWXII}.
\end{abstract}
\maketitle


\section{Notation and Preliminary Remarks}
\subsection{Notation}
The Fourier transform in this paper will be denoted by
\[
(\mathfrak{F}f)(\xi)=\widehat{f}(\xi)=\int_\R f(x)e^{-2\pi ix\xi}dx.
\]
We will identify subsets of the torus, $\mathbb{T}=\R/\Z$, with subsets of $\R$ which are invariant under translations by integers --- in particular, functions on $\T$ can be considered to be 1-periodic function on $\R$ and the Lebesgue measure we associate to $\T$ has total mass 1.  We will frequently use the operators $D^j$ and $T_k$ from $L^2(\R)$ to itself given by $D^jf=2^{j/2}f(2^j\cdot)$ and $T_kf=f(\cdot-k)$.  In particular, $D^j$ and $T_k$ are unitary maps corresponding to the dyadic dilations and integer translations.  We will let $\psi_{jk}=D^jT_k\psi=2^{j/2}\psi(2^j\cdot-k)$.
\begin{definition}
We will use the notation $m\bullet g$ to denote the function which satisfies $\widehat{m\bullet g}=m\cdot\widehat{g}$, whenever this function is well-defined.
\end{definition}

\begin{remark}
The above definition makes sense, for instance, if $g\in L^2(\R)$ and $m\in L^\infty(\R)$: then $m\cdot \widehat{g}$ is an $L^2(\R)$ function, and so $m\bullet g$ is well-defined since the Fourier transform is an isometry on $L^2(\R)$.  Though this notation is perhaps new, the above operation is actually quite common in a variety of areas of analysis.  It comes up repeatedly in the study of shift-invariant spaces, e.g. in \cite{HSWWtranslates}.  We should take a moment to reference a few of the authors who have made significant contributions to the study of shift invariant spaces upon which much of the present work rests: Helson, \cite{helson}, de Boor, DeVore, and Ron, \cite{BDR1} and \cite{BDR2}, Bownik, \cite{bowniksis}, as well as many of the references contained in those papers.  The operation $m\bullet g$ operation also comes up in the study of singular integrals --- for example, the Hilbert transform corresponds to such a map when, say, $g\in L^2(\R)$ and $m$ is (a multiple of) the function which is $1$ for positive reals and $-1$ for negative reals.\\
\end{remark}

In everything we discuss below --- unless otherwise noted --- equalities between functions are taken to be equalities in the almost everywhere sense; equalities between sets are taken to be equalities up to sets of measure zero.  All functions we will discuss below are measurable.

\subsection{Dyadic Parseval Frame MRA Wavelets}
The term wavelet can mean a variety of things in different contexts.  For our purposes, a wavelet will be a function $\psi\in L^2(\R)$ so that $\{\psi_{jk}:j,k\in\Z\}$ linearly generates $L^2(\R)$ in some way --- these are more accurately described as \emph{dyadic} wavelets since the dilations are dyadic dilations.  Classically, wavelets were always taken to be orthonormal wavelets; that is, functions $\psi$ so that $\{\psi_{jk}:j,k\in\Z\}$ corresponded to orthonormal bases of $L^2(\R)$ (e.g. the Haar wavelet, Shannon wavelet, Daubechies wavelets, and so on).  In more recent years, however, more general spanning sets have been considered.  We recall the definition of a frame:
\begin{definition}[Frame]
On a (separable) Hilbert space $\mathcal{H}$ with norm $\|\cdot\|$, a frame with frame bounds $A$ and $B$ is a collection $\{u_n\in\mathcal{H}:n\in\Z\}\subset\mathcal{H}$ so that there are numbers $A$ and $B$ with $0<A\le B<\infty$ so that for all $f\in\mathcal{H}$
\[
A\|f\|^2\le\sum_{n\in\Z}|\langle f,u_n\rangle|^2\le B\|f\|^2.
\]
\end{definition}
Of particular interest in applications are the Parseval frames:
\begin{definition}[Parseval Frame]
On a (separable) Hilbert space $\mathcal{H}$ with norm $\|\cdot\|$, a Parseval frame is a frame with frame bounds both equal to $1$.  That is, for all $f\in\mathcal{H}$
\[
\|f\|^2=\sum_{n\in\Z}|\langle f,u_n\rangle|^2.
\]
It is worth mentioning that, in older literature, Parseval frames are frequently referred to as ``normalized tight frames''.
\end{definition}

Every orthonormal basis is clearly a Parseval frame, but Parseval frames are substantially more general: for an example of a Parseval frame, one could take any two orthonormal bases $\mathcal{U}:=\{u_n:n\in\Z\}$ and $\mathcal{V}:=\{v_n:n\in\Z\}$ and consider the collection $\mathcal{W}:=\left\{w_{2n}:=\frac{1}{\sqrt{2}}u_n,w_{2n+1}:=\frac{1}{\sqrt{2}}v_n:n\in\Z\right\}$.  This collection $\mathcal{W}$ is ``redundant'' in the sense that an element $f$ of $L^2(\R)$ can be represented as (countable) linear combinations of elements of $\mathcal{W}$ in more than one way.  This redundancy is one of the key features of a Parseval frame and is frequently of theoretical and practical use; having multiple representations for the same data can aid, for example, in error correction.  As one explicit example in $L^2([0,1))$, consider the collection $\{e^{\pi inx}:n\in\Z\}$  --- this is the union of two orthonormal bases: the usual one, $\{e^{2\pi inx}:n\in\Z\}$, together with a shifted version $\{e^{\pi ix}e^{2\pi inx}:n\in\Z\}$; modulo a normalization factor, the union of these two collections form a Parseval frame for $L^2([0,1))$ --- in sampling theory, this Parseval frame is ``twice oversampled'' which would provide some protection against various undesirable effects in signal processing.  Some Parseval frames are generated by wavelets.  We define these below.
\begin{definition} 
A Parseval frame wavelet is any function $\psi$ in $L^2(\R)$ so that $\{\psi_{jk}:j,k\in\Z\}$ forms a Parseval frame for $L^2(\R)$.  We let $\mathcal{P}$ denote the class of all such Parseval frame wavelets.
\end{definition}
As a pair of simple examples, consider the two functions defined by 
\[
\widehat{\psi}_0=\chi_{[-1,-1/2)\cup[1/2,1)}
\]
and 
\[
\widehat{\psi}_1=\chi_{[-1/2,-1/4)\cup[1/4,1/2)}.
\]
In each case, the dyadic dilations of the relevant union of two intervals produce a tiling of the line; however, since the Fourier transform sends integer translation to modulation, a trivial modification of the above discussion of trigonometric polynomials in $L^2([0,1))$ quickly shows that $\psi_0$ is an orthonormal wavelet while $\psi_1$ is a Parseval frame wavelet which is \emph{not} an orthonormal wavelet.

We should mention that most Parseval frames are not generated by elements of $\p$.  In particular, Parseval frames coming from wavelets must have distinct elements.  We will prove this below but first give a useful and simple characterization of Parseval frame wavelets.
\begin{theorem}\label{caldtq}
A function $\psi\in L^2(\R)$ lies in $\mathcal{P}$ if and only if the following two conditions hold:
\begin{enumerate}
\item (Dyadic Calder\'{o}n Condition)
\[
\sum_{j\in\Z}|\widehat{\psi}(2^j\xi)|^2=1\textrm{ almost everywhere, and}
\]
\item ($t_q$ Equation) 
\[
t_q(\xi):=\sum_{j\ge 0}\widehat{\psi}(2^j\xi)\overline{\widehat{\psi}(2^j(\xi+q)}=0\textrm{ almost everywhere whenever }q\textrm{ is an odd integer.}
\]
\end{enumerate}
\end{theorem}
A proof of this result can be found in \cite[Chapter 7]{HWWave} --- we warn the reader that Theorem 7.1 contained within states the requirement that $\|\psi\|_2\ge 1$.  However, the proof of Theorem 7.1 therein actually proves the above statement; the condition $\|\psi\|_2\ge 1$ merely guarantees that the Parseval frame obtained is actually an orthonormal basis.  The Calder\'{o}n condition is a fundamental identity in the study of wavelets even in vastly more abstract settings than discussed presently.  We refer the reader to \cite{WW} for a discussion of wavelets and the Calder\'{o}n condition for a variety of dilation groups.  The $t_q$ equation is mostly a technical requirement owing to the interaction of dyadic dilations and integer translations\footnote{for example, $T_kD^j=D^jT_{2^{j}k}$, but $2^{j}k$ is not an integer if $k$ is odd and $j<0$.} and is unnecessary in the context of \cite{WW}.  The two equations given in Theorem \ref{caldtq} are extremely useful in understanding the properties of Parseval frame wavelets; for example, consider the following lemma\footnote{Parseval frames coming from wavelets actually satisfy much stronger linear independence conditions than the lack of repeated elements; see, for example, \cite{BSlindep}}:
\begin{lemma}
Suppose that $\psi\in\p$.  Then $\psi_{jk}=\psi_{mn}$ if and only if $j=m$ and $k=n$.
\end{lemma}
\begin{proof}
Suppose that $\psi\in\p$ with $\psi_{jk}=\psi_{mn}$.  Then
\[
2^{j/2}\psi(2^jx-k)=2^{m/2}\psi(2^mx-n).
\]
Without loss of generality, we assume that $m\le j$.  Then by applying $D^{-j}$ to the previous equality and then $T_{-2^{j-m}n}$, the above is possible if and only if, for $\ell=j-m\ge 0$,
\[
\psi(x-k+2^{-\ell}n)=2^{-\ell/2}\psi(2^{-\ell} x)
\]
Taking Fourier transforms of both sides, one gets
\[
\widehat{\psi}(\xi)e^{-2\pi i(k-2^{-\ell}n)\xi}=2^{\ell/2}\widehat{\psi}(2^\ell \xi).
\]
By the Dyadic Calder\'{o}n Condition, one has for almost every $\xi$,
\[
1=\sum_{p\in\Z}|\widehat{\psi}(2^p\xi)|^2=\sum_{p\in\Z}\left|\widehat{\psi}(2^p\xi)e^{-2\pi i(k-2^{-\ell}n)\xi}\right|^2=\sum_{p\in\Z}\left|2^{\ell/2}\widehat{\psi}(2^{\ell+p} \xi)\right|^2=2^\ell\sum_{p\in\Z}|\widehat{\psi}(2^{p+\ell}\xi)|^2=2^\ell.
\]
Thus $j$ and $m$ must be equal, giving $\widehat{\psi}(\xi)e^{-2\pi i(k-n)\xi}=\widehat{\psi}(\xi)$; since $\psi$ is nonzero on a set of positive measure, this forces $e^{-2\pi i(k-n)\xi}=1$, i.e. $n=k$.\qed
\end{proof}

The class $\p$ is known to have an extremely complicated structure made up of many different subclasses; see \cite{SSW}.  For example, the class of orthonormal wavelets --- those for which $\psi_{jk}$ is an orthonormal basis for $L^2(\R)$ --- can be characterized as those $\psi\in\p$ for which $\|\psi\|_2=1$.  We seek to understand one of the subclasses of $\p$ which is more complicated than the family of orthonormal wavelets but still nice enough to be of serious practical interest.  To that end, we recall the following definitions:
\begin{definition}
A closed subspace $V$ of $L^2(\R)$ is said to be shift invariant if $T_kV\subset V$ for any integer $k$.  Given a collection of nonzero functions $\{\phi_n:n\in\Z\}$, the space $\left\langle \{\phi_n:n\in\Z\}\right\rangle$ is the intersection of all shift-invariant subspaces containing $\{\phi_n:n\in\Z\}$.  In the case that $\{\phi_n:n\in\Z\}$ consists of exactly one element $\phi$, we denote the shift-invariant space generated by $\phi$ as $\langle\phi\rangle$ rather than the more cumbersome $\langle\{\phi\}\rangle$ and refer to it as a principal shift-invariant space; in particular, one has
\[
\langle\phi\rangle=\overline{\textrm{span}\{\phi(x-k):k\in\Z\}}.
\]
\end{definition}

For any $\phi\neq 0$, it is not too hard to see that functions $\phi(x-k)$ for integer $k$ form a linearly independent set --- $\sum_{finite}a_k\phi(x-k)$ has Fourier transform equal to a trigonometric polynomial times $\widehat{\phi}$, and nonzero trigonometric polynomials cannot be zero on sets of positive measure.  Thus, in particular, $\langle\phi\rangle$ is infinite dimensional as a subspace of $L^2(\R)$.  Nonetheless, $\langle\phi\rangle$ is, in some sense, one dimensional since it is completely determined by a single function.  The so-called dimension function is the appropriate notion of dimension in the context of shift-invariant spaces.  There are a number of heuristic approaches to construction this function, although they are often difficult to implement or require some advanced machinery; we describe an elementary construction.

We make the following observations: when $V$ is a shift-invariant space which is closed in $L^2(\R)$, so is its orthogonal $V^\perp$ since $T_k$ is unitary for every $k\in\Z$.  Secondly, by Zorn's Lemma, for every shift-invariant subspace $V$ we can choose a collection $\Phi:=\{\phi_i:i\in I\}$ for a countable index set $I$ such that $V=\bigoplus_{i\in I}\langle\phi_i\rangle$ so that $\mathcal{B}_{\phi_i}:=\{\phi_i(\cdot-k):k\in\Z\}$ is a Parseval frame for $\langle\phi_i\rangle$.  Note that, as a result, $\mathcal{B}_\Phi:=\bigcup_{i\in I}\mathcal{B}_{\phi_i}$ is then a Parseval frame for $V$ since the $\langle\phi_i\rangle$ are orthogonal to one another --- in fact $\mathcal{B}_\Phi$ is a special kind of Parseval frame: a so-called semiorthogonal Parseval frame.
\begin{lemma}
Suppose that both $\Phi=\{\phi_i:i\in I\}$ and $\Psi=\{\psi_j:j\in J\}$ are both Parseval frames (not necessarily semiorthogonal) which generate the same space, i.e. $\left\langle\Phi\right\rangle=\left\langle\Psi\right\rangle$, then $\sum_{i\in I}[\phi_i,\phi_i](\xi)=\sum_{j\in J}[\psi_j,\psi_j](\xi)$ almost everywhere.
\end{lemma}
Before proving this lemma, we recall the bracket function:
\begin{definition}
Given any two functions $\phi,\psi\in L^2(\R)$, we define the bracket $[\phi,\psi]$ to be the 1-periodic function,
\[
[\phi,\psi](\xi):=\sum_{k\in\Z}\widehat{\phi}(\xi+k)\overline{\widehat{\psi}(\xi+k)}.
\]
Frequently, one denotes $[\phi,\phi]$ by $p_\phi$.
\end{definition}
\proof
We first note that the Parseval frame property for $\Phi$ can be expressed by $f=\sum_{i\in I}[f,\phi_i]\bullet\phi_i$ for each $f\in \langle\Phi\rangle=\langle\Psi\rangle$.  Using this to write each $\phi_i$ in terms of elements of $\Psi$ and each $\psi_j$ in terms of elements of $\Phi$, we have that
\begin{align*}
\sum_{i\in I}[\phi_i,\phi_i](\xi)&=\sum_{i\in I}[\phi_i,\sum_{j\in J}[\phi_i,\psi_j]\bullet\psi_j]=\sum_{i\in I,j\in J}\sum_{k\in\Z}\widehat{\phi_i}(\xi+k)\overline{[\phi_i,\psi_j](\xi)\widehat{\psi_j}(\xi+k)}\\
&=\sum_{i\in I,j\in J}|[\phi_i,\psi_j](\xi)|^2=\sum_{j\in J}[\psi_j,\psi_j](\xi).
\end{align*}

\qed
\begin{definition}
In the context of the preceding lemma, we may define, up to null sets, the dimension function of a shift-invariant space $V$:
\[
\textrm{dim}_V(\xi):=\sum_{i\in I}[\phi_i,\phi_i](\xi)
\]
for \emph{any} choice of a Parseval frame generating set $\Phi=\{\phi_i:i\in I\}$ for $V$.
\end{definition}
\begin{remark}
There is no harm in taking $\Phi$ to be semiorthogonal as mentioned in the paragraph preceding the above theorem.  Then $[\phi_i,\phi_i](\xi)=\chi_{S_i}$, for some measurable set $S_i$.  This proves the following theorem.
\end{remark}
\begin{theorem}
\begin{enumerate}
\item The dimension function takes values in $\{0,1,2,3,...\}\cup\{\infty\}$ almost everywhere.
\item For any $\phi\in L^2(\R)$ which is not the 0 function, $\textrm{dim}_{\langle\phi\rangle}(\xi)=0\textrm{ or }1$ almost everywhere.
\end{enumerate}
\end{theorem}

\begin{definition}
An orthonormal multi-resolution analysis (MRA) is given by a countable family $V_j$ of closed subsets of $L^2(\R)$ with the following properties:
\begin{enumerate}
\item $V_j\subset V_{j+1}$ for all $j\in\Z$
\item $\phi\in V_j$ if and only if $\phi(2\cdot)\in V_{j+1}$
\item $\bigcap_{j\in\Z}V_j=\{0\}.$\footnote{Actually parts 1, 2, and 5 of this definition guarantee this (\cite[Chapter 2, Theorem 1.6]{HWWave}) but it is an important fact, so we include it in the definition.}
\item $\overline{\bigcup_{j\in\Z}V_j}=L^2(\R)$
\item There is a function $\phi\in V_0$ so that $V_0=\langle\phi\rangle$ and $\mathcal{B}_\phi:=\{\phi(\cdot-k):k\in\Z\}$ is an orthonormal basis for $V_0$.
\end{enumerate}
\end{definition}
The function $\phi$ in the definition above is referred to as the \emph{scaling function} for the MRA.  Given an orthonormal MRA, one can construct many orthonormal wavelets.  The two key objects in this construction are the so-called low- and high-pass filters.  We describe the construction and these filters briefly.  Since $V_{-1}\subset V_0$, we know that $\phi(x/2)\in V_0$, so that there is a sequence $(a_k)_{k\in\Z}\in\ell^2(\Z)$ satisfying
\[
\phi(x/2)=\sum_{k\in\Z}a_k\phi(x-k)
\]
in the $L^2(\R)$ sense.  This, together with properties of the Fourier transform, gives us the \emph{two-scale equation},
\begin{equation}
\widehat{\phi}(2\xi)=\left[\frac{1}{2}\sum_{k\in\Z}a_ke^{-2\pi ik\xi}\right]\widehat{\phi}(\xi):=m_0(\xi)\widehat{\phi}(\xi),\tag{Two-Scale Equation}\label{twoscaleequation}
\end{equation}
where $m_0$ is the 1-periodic function in the brackets in the equality above; the function $m_0$ is in $L^2(\T)$, and we refer to it as the \emph{low-pass filter associated to $\phi$}.  It is not terribly difficult to deduce that if $\{\phi(x-k):k\in\Z\}$ is an orthonormal set in $L^2(\R)$, then 
\[
\sum_{k\in\Z}|\widehat{\phi}(x+k)|^2=1.
\]
This, together with the Two-Scale Equation give
\begin{align*}
1=\sum_{k\in\Z}|\widehat{\phi}(2\xi+k)|^2&=\sum_{k\in\Z}|\widehat{\phi}(2\xi+2k)|^2+\sum_{k\in\Z}|\widehat{\phi}(2\xi+2k+1)|^2\\
&=\sum_{k\in\Z}|\widehat{\phi}(2(\xi+k))|^2+\sum_{k\in\Z}|\widehat{\phi}(2(\xi+1/2+k))|^2\\
&=|m_0(\xi)|^2\sum_{k\in\Z}|\widehat{\phi}(\xi+k)|^2+|m_0(\xi+1/2)|^2\sum_{k\in\Z}|\widehat{\phi}(\xi+1/2+k)|^2\\
&=|m_0(\xi)|^2+|m_0(\xi+1/2)|^2.
\end{align*}
This result is known as the \emph{Smith--Barnwell Equation}, a necessary condition on a low-pass filter:
\begin{equation}
|m_0(\xi)|^2+|m_0(\xi+1/2)|^2=1.\label{smithbarnwell}\tag{Smith--Barnwell Equation}
\end{equation}
Now, since $V_0\subset V_1$, we may define the function $\psi\in V_1\cap V_0^\perp$ by $\widehat{\psi}(2\xi)=e^{2\pi i\xi}\nu(2\xi)\overline{m_0(\xi+1/2)}\widehat{\phi}(\xi)$, where $\nu$ can be any measurable, 1-periodic, unimodular function.  $\psi$ can be shown to be an orthonormal wavelet.  We define by $m_1$ the 1-periodic function $m_1(\xi)=e^{2\pi i\xi}\nu(2\xi)\overline{m_0(\xi+1/2)}$ and refer to $m_1$ as the \emph{high-pass filter associated to $\psi$}.  Note that, as a result of the function $\nu$, there are many different choices of $m_1$ and $\psi$ for a given $\phi$ which give different wavelet functions.  One fundamental fact about these filters is that $m_0,m_1$ form a Smith--Barnwell pair.  More precisely, this means that the matrix 
\[
\begin{pmatrix}
m_0(\xi) & m_0(\xi+1/2)\\
m_1(\xi) & m_1(\xi+1/2)
\end{pmatrix}
\]
is unitary for almost every $\xi$.  More precise details of the above may be found in \cite[Chapter 2, Section 2]{HWWave}.
\begin{definition}
A nonzero function $\psi\in L^2(\R)$ is an orthonormal MRA wavelet if $\psi_{jk}$ forms an orthonormal basis for $L^2(\R)$ and it can be associated to an MRA as described in the procedure above.
\end{definition}
Both the Haar wavelet, which has scaling function $\phi_H=\chi_{[0,1)}$, and the Shannon wavelet, which has scaling function $\widehat{\phi_S}=\chi_{[-1/2,1/2)}$ are prototypical examples of orthonormal MRA wavelets. 

One of the main goals of the present work will be to understand the Parseval frame analogs of these orthonormal MRA wavelets.  We will give a precise definition of these analogs shortly, but we wish to point out the following major difference: if $\phi(x-k)$ is a Parseval frame for $\langle\phi\rangle$ but not an orthonormal basis, then one deduces that
\[
\sum_{k\in\Z}|\widehat{\phi}(\xi+k)|^2=\chi_{C_\phi+\Z},
\]
where $C_\phi=\textrm{supp }\widehat{\phi}$; in particular, this function will be zero on a set with positive measure.  If $D^{-1}\phi\in\langle\phi\rangle$, one can still construct a low-pass filter for $\phi$ via the Two-Scale Equation $\widehat{\phi}(2\xi)=m_0(\xi)\widehat{\phi}(\xi)$, but  the values of $m_0$ when $\xi$ is not in $C_\phi+\Z$ are irrelevant.

The authors of \cite{PSWXI} and \cite{PSWXII} began a systematic study of these issues and developed a new notion of low-pass filters, which they called generalized low-pass filters, from which to build Parseval frame analogs of MRA wavelets.  More precisely, they defined generalized low-pass filters to be 1-periodic functions which satisfy the Smith--Barnwell equation on the entirety of $\T$.  They then defined generalized scaling functions (actually, they use the somewhat clumsy terminology ``pseudoscaling function'') to be those nonzero functions $\phi\in L^2(\R)$ for which there is a generalized low-pass filter $m$ satisfying 
\[
\widehat{\phi}(2\xi)=m(\xi)\widehat{\phi}(\xi).
\]
The pair $(\phi,m)$ of generalized low-pass filter and generalized scaling function is then used to construct Parseval frame analogs of orthonormal MRA wavelets.

However, this requirement is a bit overly strong: since the low-pass filter is only a relevant quantity on $C_\phi+\Z$, it is not necessary for the low-pass filter to satisfy Smith--Barnwell everywhere.  Moreover, there is a stronger tie between the low-pass filter and the support of $\widehat{\phi}$ that deserves a bit more care to flush out.  Thus we will change, slightly, the definition of low-pass filters from \cite{PSWXI} and \cite{PSWXII} as well as the definition of generalized scaling functions.  One subtle difference between the present work and \cite{PSWXI} and \cite{PSWXII} is the following: there, the authors begin with a filter and associate to it a scaling function; however, we will begin with a scaling function and then associate to it a family of associated low-pass filters; after selecting the desired filter, one can then construct a wavelet essentially as in \cite{PSWXI} and \cite{PSWXII}.

We will discuss the technical requirements of generalized low-pass filters and generalized scaling functions later on (in Section \ref{glpf}).  In the present article, we merely refer to these functions as low-pass filters and scaling functions rather than generalized low-pass filters and generalized scaling functions.

To wit, a (generalized) scaling function together with its (generalized) low-pass filter can be used to create analogs of MRA wavelets.  As we described above, there may be, for a given scaling function, many distinct choices of $m$ which differ on sets of positive measure.  The ideas related to and constructions of orthonormal MRA wavelets then lead us to the following two possible definitions for Parseval Frame MRA wavelets.
\begin{definition}

\begin{enumerate}
\item Denote by $\pmra_1$ those functions $\psi\in\p$ so that there is a generalized scaling function $\phi$ with associated generalized low-pass filter $m$ so that
\[
\widehat{\psi}(2\xi)=e^{2\pi i\xi}\nu(2\xi)\overline{m(\xi+1/2)}\widehat{\phi}(\xi),
\]
for some measurable, unimodular, 1-periodic function $\nu$.
\item Suppose that $\psi\in\p$.  Let $V_0(\psi)=\left\langle\{D^j\psi:j<0\}\right\rangle$.  Then we say that $\psi\in\pmra_2$ if $V_0(\psi)$ is a principal shift-invariant space.  That is, there is some $\phi$ so that $V_0(\psi)=\langle\phi\rangle$.  Please note that while this condition says that translates of $\phi$ linearly span $V_0(\psi)$, there is no further constraint on how these translates generate $V_0(\psi)$. For example, it need not be the case that translates of $\phi$ form a Parseval frame for $\langle\phi\rangle$.
\end{enumerate}
\end{definition}
Both of these definitions seem reasonable from the point of view of orthonormal MRA wavelets.  In fact they are the same.
\begin{theorem}\label{pmra1=pmra2}
The above two definitions are equivalent, i.e. $\pmra_1=\pmra_2$.  If $\phi$ is the scaling function associated to $\psi$ from the first definition, then $V_0(\psi)=\langle\phi\rangle$.  If $V_0(\psi)$ is a principal shift-invariant space for some $\psi\in\pmra_2$, then one can pick a $\phi$ so that $V_0(\psi)=\langle\phi\rangle$ and $\phi$ is the scaling function for $\psi$ in the sense of the definition of $\pmra_1$.
\end{theorem}
This result is a consequence of \cite{PSWXII}, but the proof therein is quite complicated.  As a service to the reader, we provide a shorter version of their proof in later sections.
\begin{notation}
In lieu of the preceding theorem, we let $\pmra:=\pmra_1=\pmra_2$.
\end{notation}

Owing to many of the non-uniquenesses hiding behind the scenes, there may be, for a given $\psi$, many $\phi$'s so that $V_0(\psi)=\langle\phi\rangle$.  These choices may allow one to vary how $V_0$ sits inside $D^1V_0$, which can dramatically impact the characteristics of the Parseval Frame MRA wavelet.  One of our first goals is to attempt to understand what happens upon varying the members of $(\psi,\phi,m_0,m_1)$ according to choices afforded us.

The authors of \cite{PSWXI} and \cite{PSWXII} discussed the following question: given a generalized scaling function $\phi$ and a 1-periodic function $m_0\in L^2(\T)$ satisfying the two-scale equation, i.e. $\widehat{\phi}(2\xi)=m_0(\xi)\widehat{\phi}(\xi)$, when is $\phi$ the scaling function for a wavelet $\psi\in\pmra$?  In their construction, they made ``arbitrary'' choices at various stages (e.g. they ignored the issue of the support set of the generalized low-pass filter by requiring the Smith--Barnwell equation to hold on the entire torus); those authors were understandably content to produce a $\pmra$ wavelet from a pair $(\phi,m_0)$ without studying the class of all possible outcomes of their choices.  This omission nagged the authors of the present work, and so this article will attempt to retread much of the work in \cite{PSWXI} and \cite{PSWXII} while giving a more careful (and, unfortunately, occasionally tedious) treatment of the impact of various choices from their construction.  This will occupy most of the second section of this paper.

This attention to detail may seem overly pedantic but actually comes with two remarkable conclusions.  In the third section of this article, we will show that the class of pairs $(\phi,m_0)$ associated to wavelets in $\pmra$ are actually the orthogonal projections (in the $L^2(\R)$ sense) of some pair $(\phi^*,m_0^*)$ where $\phi^*$ is a so-called maximal scaling function and $m_0^*$ is its associated low-pass filter (this means that $\langle\phi^*\rangle$ is not properly contained in any principal shift-invariant space) --- in particular, $\phi=\chi_{S+\Z}\bullet\phi^*$ for some set $S\subset\R$ and maximal scaling function $\phi^*$.  We will also be able to give a complete characterization of those sets $S$ so that, if $(\phi^*,m_0^*)$ is a maximal scaling function, then $\chi_{S}\bullet\phi^*$ is also associated to a $\pmra$ wavelet.

\section{Scaling Functions for $\pmra$}
\begin{definition}
Given a function $\phi\in L^2(\R)$, we make the following notations.  
\begin{enumerate}
\item Recall that $p_\phi$ denotes the weight function corresponding to $\phi$ and is given by
\[
p_\phi(\xi)=[\phi,\phi](\xi)=\sum_{k\in\Z}|\widehat{\phi}(\xi+k)|^2.
\]
Note that $p_\phi$ is 1-periodic.
\item Denote by $C_\phi$ the support of $\widehat{\phi}$, i.e.
\[
C_\phi=\{\xi\in\R:|\widehat{\phi}(\xi)|>0\}.
\]
\item Denote by $S_\phi$ the support of $p_\phi$ so that
\[
S_\phi=C_\phi+\Z.
\]
\item Denote by $\widetilde{S}_\phi$ the set given by
\[
\stilde_\phi=C_\phi+\frac{1}{2}\Z=S_\phi\cup\left(S_\phi+\frac{1}{2}\right).
\]
\end{enumerate}\label{basicdefs}
\end{definition}
\begin{remark}
Observe that
\[
\frac{1}{2}S_\phi=\frac{1}{2}C_\phi+\frac{1}{2}\Z=\left(\frac{1}{2}C_\phi+\Z\right)\cup\left(\frac{1}{2}C_\phi+\frac{1}{2}+\Z\right)
\]
and each of $\frac{1}{2}S_\phi,S_\phi\cap(S_\phi+1/2),\stilde-(1/2S_\phi),\stilde-(S_\phi\cap(S_{\phi}+1/2))$, and $\T-\stilde_\phi$ are invariant under translation by elements of $1/2\Z$.  Translations by half-integers arise immediately as a consequence of the Smith--Barnwell equations, and so the set $\widetilde{S}_\phi$ is a relatively natural object to study.
\end{remark}
\subsection{Unimodular Functions and Some of Their Properties}
\begin{definition} 
We denote by $\mathcal{A}$ the group of functions $\alpha$ on $\R$ which are unimodular in the sense that $|\alpha|\equiv 1$.  We denote by $\mathcal{U}$ the subgroup of $\mathcal{A}$ which is 1-periodic.  We let $\mathcal{A}\bullet\phi$ denote the orbit of $\phi$ under the left action of $\mathcal{A}$.
\end{definition}

Obviously, the mapping $\phi\mapsto\alpha\bullet\phi$ is unitary on $L^2(\R)$ and $\phi_1\in \mathcal{A}\bullet\phi$ if and only if $|\widehat{\phi_1}|=|\widehat{\phi}|$.

\begin{definition}
As defined previously, the maps $D^j$ correspond to the $L^2$-normalized operator $f\mapsto 2^{j/2}f(2^j\cdot)$.  We denote by $D^{j,\infty}$ the $L^\infty$-normalized map (i.e. $\|D^{j,\infty}f\|_\infty=\|f\|_\infty$) given by $D^{j,\infty}f=f(2^j\cdot)$.  Note that $D^{j,\infty}$ is still bounded on $L^2(\R)$, it is simply a non-unitary dilation on $L^2(\R)$.
\end{definition}

Clearly $D^{j,\infty}$ is an automorphism of $\mathcal{A}$ for all $j\in\Z$.

\begin{definition}
For each $j\in\Z$, write $\mathcal{U}_j:=D^{j,\infty}\mathcal{U}\equiv\{\alpha\in\mathcal{A}:\alpha\textrm{ is }1/2^j\textrm{ periodic}\}.$  In particular, $\mathcal{U}_0=\mathcal{U}$.
\end{definition}
Observe that for $\alpha\in\A$.
\[
\mathfrak{F}\left(D^{-j}(\alpha\bullet\phi)\right)=D^{j}(\alpha\widehat{\phi})=D^{j,\infty}\alpha D^j\widehat{\phi}=(D^{j,\infty}\alpha)(\mathfrak{F}(D^{-j}\phi)),
\]
so that
\[
D^{j}(\alpha\bullet\phi)=(D^{-j,\infty}\alpha)\bullet (D^j\phi)
\]
\begin{definition}
For $\alpha\in\A$, we note that 
\[
\delta_\alpha:=\frac{D^{1,\infty}\alpha}{\alpha}
\]
is also in $\A$.  Moreover, the map $\alpha\mapsto\delta_\alpha$ defines a homomorphism $\delta$ from $\A$ to itself with $\delta(\U)\subset\U$ since $\U_1=D^{1,\infty}\U\subset\U$.  Let $\M:=\{\alpha\in\A:\delta_\alpha\in\U\}$ and observe that $\M$ is a subgroup of $\A$.
\end{definition}
\begin{remark}
The importance of $\M$ arises from the fact that if $\alpha\in\M$ and $\psi=D^{-1}\phi$, then $D^{-1}(\alpha\bullet\phi)=(D^{1,\infty}\alpha)\bullet D^{-1}\phi$.  When $\phi$ is \emph{reductive} in the sense that there is a function $m_0$ on $S_\phi$ such that $\frac{1}{\sqrt{2}}D^{-1}\phi=m_0\cdot\phi\in\langle\phi\rangle$, or, equivalently,
\[
\widehat{\phi}(2\xi)=\left\{
        \begin{array}{ll}
            m_0(\xi)\widehat{\phi}(\xi) & \quad \xi\in S_\phi \\
            0 & \quad \xi\notin S_\phi
        \end{array}
    \right.,
\]
then
\[
\frac{1}{\sqrt{2}}D^{-1}(\alpha\bullet\phi)=(\delta_\alpha m_0)\bullet(\alpha\bullet\phi)=(\nu m_0)\bullet(\alpha\bullet\phi),
\]
so that $\alpha\bullet\phi$ is also reductive.  Note in passing that $m_0$ is the unique 1-periodic function on $S_\phi=C_\phi+\Z$ for which $m_0(\xi)=\frac{\widehat{\phi}(2\xi)}{\widehat{\phi}(\xi)}$ when $\xi\in C_\phi$.  Also, $\frac{1}{2}C_\phi\subset C_\phi$, so by the remark after Definition \ref{basicdefs}, we have that $\frac{1}{2}S_\phi\subseteq \stilde_\phi$.
\end{remark}
\begin{lemma}\label{lemma1.4}
One has that $\delta(\M)=\U$.  In particular, for each $\nu\in\U$, each $N>0$, and each $\alpha_0:I_0=[-N,-N/2)\cup(N/2,N]\mapsto\partial\mathbb{D}$, there is a unique $\alpha\in\M$ for which $\alpha\big|_{I_0}=\alpha_0$, and $\nu=\delta_\alpha$.  As a result, letting $\M_\nu:=\delta^{-1}(\nu)=\{\alpha\in\M:\delta_\alpha=\nu\}$, we have that $\M_\nu$ and $\U$ have cardinality equal to the power set of $\R$.
\end{lemma}
\begin{proof}
For $k\in\Z$, let $I_k=2^kI_0$.  Then $\R$ is the disjoint union of $\{0\}$ and the sets $I_k$ for $k\in\Z$.  So, each $\alpha\in\A$ is uniquely determined by its values on these sets.  Since $I_{k+1}=2I_k$ for each $k$, one has for $n\in\mathbb{N}\cup\{0\}$ that $I_n=\frac{1}{2}I_{n+1}$ and $I_{-n}=2I_{-n-1}$.  Put $\alpha\big|_{I_0}=\alpha_0$ and suppose inductively that we have defined $\alpha_k:=\alpha\big|_{I_k}$ for $|k|\le n$.  Then define $\alpha_{n+1}$ and $\alpha_{-n-1}$ by $\alpha_{n+1}(\xi)=(\nu\alpha_n)(\xi/2)$ for $\xi\in I_{n+1}$ and $\alpha_{-n-1}=(\nu\alpha_{-n})(2\xi)$ for $\xi\in I_{-n-1}$.  Obviously this inductive process gives us the unique $\alpha\in\M_\nu$ for which one has $\alpha\big|_{I_0}=\alpha_0$.  Taking into account the cardinality of the set of measurable functions from $I_0\rightarrow\partial\mathbb{D}=\{z\in\Z:|z|=1\}$, we get the desired conclusion.\qed
\end{proof}
\begin{remark}
Note that the trigonometric polynomials from $\T$ given by $e_k=e^{-2\pi ikx}$ are also the continuous homomorphisms from $\T$ into $\partial\mathbb{D}$ and hence $D^{1,\infty}e_k=e_k^2$ and $\delta_{e_k}=e_k$.  This isn't very helpful since $\{c_0e_k:c_0\in\partial\mathbb{D},k\in\Z\}$ is only a tiny subgroup of $\U$.  We can imbed $\U$ into $L^\infty(\R)$ by associating $\nu\in\U$ with its almost everywhere equivalence class in $L^\infty(\R)$.  The image $\overline{\U}$ of $\U$ under this imbedding is a closed subset of $L^\infty(\R)$.  It is very likely that $(\overline{\U},\|\cdot\|_\infty)$ is non-separable and that the groups $\U$ and $\overline{\U}$ don't have countable generating sets.  This difficulty lies at the core of such questions as ``nailing down'' $\U\cap\M_\nu=\{\nu'\in\U:\delta_{\nu'}=\nu\}$ for $\nu\in\U$ and ``constructing'' a set of coset representations for $\U/\U_1$ --- obviously every $1/2$-periodic function is also 1-periodic so $\U_1$ is a subgroup of $\U$.  As we shall see, these questions arise naturally in the theory of scaling functions for MRA wavelets.
\end{remark}

\subsection{Scaling Functions and Low-Pass Filters}
We recall that the authors of \cite{PSWXI} and \cite{PSWXII} proceeded by focusing their attention on the filters associated to an $\pmra$ wavelet.  More precisely, they define the collection of ``admissible'' filters and then construct the $\pmra$ wavelet from such a filter.  We instead will develop a theory of ``admissible'' pairs $(\phi,m_0)$ of scaling functions and their associated low-pass filter.  We begin by defining explicitly our notion of admissible pairs $(\phi,m_0)$.  We will then explain why they are equivalent to the previous work in \cite{PSWXI} and \cite{PSWXII}.

\subsubsection{Scaling Functions and their Properties}
\begin{definition}
For $\phi\in L^2(\R)$ we say that $\phi$ belongs to the set $\s$ of (dyadic) scaling functions if the following hold:
\begin{enumerate}
\item ($\s1$) $|\widehat{\phi}|$ is dyadically continuous at $0$ in the sense that
\[
\lim_{n\rightarrow\infty}|\widehat{\phi}(2^{-n}\xi)|=1\textrm{ for almost every }\xi\in\R.
\]
Note that this implies that for almost every $\xi\in\R$, there exists $n_\xi$ so that $2^{-n}\xi\in C_\phi$ for all $n\ge n_\xi$.
\item ($\s2$) $\phi$ is reductive so $\frac{1}{2}C_\phi\subseteq C_\phi$ and $\frac{1}{\sqrt{2}}D^{-1}\phi=m_0\bullet\phi$ with $m_0$ the unique 1-periodic function on $S_\phi=C_\phi+\Z$ for which
\begin{equation}\label{s2condition}
\widehat{\phi}(2\xi)=\left\{
        \begin{array}{ll}
            m_0(\xi)\widehat{\phi}(\xi) & \quad \xi\in S_\phi \\
            0 & \quad \xi\notin S_\phi
        \end{array}
    \right..
\end{equation}
Equivalently, $m_0(\xi)=\frac{\widehat{\phi}(2\xi)}{\widehat{\phi}(\xi)}$ on $C_\phi$, $|m_0(\xi)|>0$ on $\frac{1}{2}C_\phi$ and $m_0=0$ on $C_\phi-\frac{1}{2}C_\phi$.
\item ($\s3$) $|m_0|\le 1$ on $S_\phi$ and on $S_\phi\cap(S_\phi+\frac{1}{2})$, the function $m_0$ satisfies the Smith--Barnwell Equation, $|m_0(\xi)|^2+|m_0(\xi+1/2)|^2=1$.
\end{enumerate}
We call $m_0$ the low-pass filter for $\phi$ and use the notation $(\phi,m_0)\in\s$ to mean $\phi\in\s$ and $m_0$ is its low-pass filter.
\end{definition}
We now enumerate some elementary properties of $(\phi,m_0)\in\s$.
\begin{enumerate}
\item By iteration of (\ref{s2condition}), we have, on $C_\phi$
\begin{equation}\label{s2iteration}
\widehat{\phi}(\xi)=\left(\prod_{j=1}^nm_0(2^{-j}\xi)\right)\widehat{\phi}(2^{-n}\xi).
\end{equation}
By taking the modulus of both sides and letting $n\rightarrow\infty$ and using the dyadic continuity condition $(\s1)$, one has
\begin{equation}
|\widehat{\phi}(\xi)|=\prod_{j=1}^\infty|m_0(2^{-j}\xi)|.
\end{equation}
It follows that if $\widetilde{m}_0$ is any extension of $m_0$ from $S_\phi$ to $\T$ which satisfies Smith--Barnwell on $\T$, then the previous two formulae for $\widehat{\phi}$ hold when $m_0$ is replaced by $\widetilde{m}_0$.  If $\phi$ happens to be continuous at $0$, then one has
\begin{equation}\label{continuousphi}
\widehat{\phi}(\xi)=\widehat{\phi}(0)\prod_{j=1}^\infty\widetilde{m}_0(2^{-j}\xi)
\end{equation}
with $|\widehat{\phi}(0)|=1$ by ($\s2$) and $m_0(0)=\widehat{\phi}(2\cdot0)/\widehat{\phi}(0)=1$.
\item For $\alpha\in\M$ with $\nu=\delta_\alpha$, one has $|\alpha\widehat{\phi}|=|\widehat{\phi}|$, so $C_{\alpha\bullet\phi}=C_\phi$ and $S_{\alpha\bullet\phi}=S_\phi$.  Obviously, $\alpha\bullet\phi$ satisfies ($\s1$) and also satisfies ($\s2$) with $\frac{1}{\sqrt{2}}D^{-1}(\alpha\bullet\phi)=(\nu m_0)\bullet(\alpha\bullet\phi)$.  Since $|\nu|\equiv 1$, one concludes that $\nu m_0$ satisfies ($\s3$).  Thus $(\alpha\bullet\phi,\nu m_0)\in\s$ and we can express this by saying that the group $\M$ acts on $\s$ by 
\[
\alpha\cdot(\phi,m_0)=(\alpha\bullet\phi,\delta_\alpha m_0).
\]
From Lemma \ref{lemma1.4}, we have $\M_\nu=\delta^{-1}(\nu)\neq\emptyset$ for each $\nu\in\U$.  In particular, for $\nu\equiv 1$, $(\alpha\bullet\phi,m_0)\in\s$, but, in contrast to (\ref{continuousphi}), $\alpha\bullet\phi$ is obviously not determined by $m_0$.  Note that, for $\alpha,\alpha'\in\M$, we have $\alpha\bullet\phi=\alpha'\bullet\phi$ if and only if $\alpha\big|_{C_\phi}=\alpha'\big|_{C_\phi}$.  Then $\alpha\bullet\phi\in\langle\phi\rangle$ if and only if $(\alpha-\mu)\big|_{C_\phi}=0$ for some $\mu\in\M$, in which case we have
\[
\alpha\bullet\phi=\mu\bullet\phi\in\U\bullet\phi
\]
and
\[
\delta_\alpha=\delta_\mu\textrm{ on }S_\phi,
\]
which finally tells us $\alpha\cdot(\phi,m_0)=(\mu\bullet\phi,\delta_\mu m_0)$.  As we mentioned in the remark at the end of the previous section, we don't have very good control over $\{\nu\in\U:\U\cap\delta^{-1}\nu\neq\emptyset\}$.
\end{enumerate}
\subsubsection{Low-Pass Filters and their Properties}\label{glpf}
\begin{definition}
Let $C\subseteq\R$ have the property that
\begin{equation}\label{dyadicpartition}
\R=\bigcup_{j\in\Z}2^{j}C.
\end{equation}
For $m_0$ a 1-periodic function on $S=C+\Z\subseteq\T$, $m_0$ belongs to the set $LP(C)$ of low-pass filters associated with $C$ if it has the following properties:
\begin{enumerate}
\item (LP1) $\lim_{n\rightarrow\infty}\prod_{j=0}^\infty|m_0(2^{-j}\xi)|=1$ on $S$.
\item (LP2) $|m_0|>0$ on $C/2$ and $0$ on $C-C/2$.
\item (LP3) $|m_0|\le1$ on $S$ and the Smith--Barnwell equation holds on $S\cap(S+1/2)$.
\end{enumerate}
\end{definition}
\begin{remark}
This is, essentially, a reformulation of the definition from \cite{PSWXI} and \cite{PSWXII}.  We should also mention that the problem of characterizing low-pass filters has been studied by other authors, e.g. \cite{cohen} and \cite{PapSW}. When $(\phi,m_0)\in\s$, we know from the previous section that $C_\phi$ satisfies (\ref{dyadicpartition}) and $m_0\in LP(C_\phi)$.  Conversely, when $m_0\in LP(C)$, it is shown in \cite{PSWXI} that (LP1) and (LP3) imply the existence of $(\phi_{|m_0|},|m_0|)$ for which we can assume 
\[
\widehat{\phi_{|m_0|}}(\xi)=\prod_{j=1}^\infty|m_0(2^{-j}\xi)|\textrm{ for all }\xi\in\R,
\]
and that $C=C_{\phi_{|m_0|}}$.  We can then choose $\nu\in\U$ for which $m_0=\nu|m_0|$.  By Lemma \ref{lemma1.4}, $\M_\nu=\delta^{-1}(\nu)\neq\emptyset$ and, from the previous subsection, $(\alpha\bullet\phi_{|m_0|},m_0)\in\s$ for each $\alpha\in\M_\nu$.  Now suppose we have some $\phi\in L^2(\R)$ for which $(\phi,m_0)\in\s$.  From the previous subsection, $C=C_\phi$ and $|\widehat{\phi}|=\widehat{\phi_{|m_0|}}$.  We will show that there is a unique $\alpha\in\M_\nu$ for which $(\phi,m_0)=\alpha\cdot(\phi_{|m_0|},|m_0|)=(\phi,m_0)$.  To see this, define $\alpha$ on $C=C/2\cup(C\backslash C/2)$ to be $\frac{\widehat{\phi}}{|\phi|}$.  By (\ref{dyadicpartition}), we know that $\R$ is the disjoint union of $C/2$ and the sets $B_n=2^n(C\backslash C/2)$, $n\in\{0\}\cup\N$.  We define $\alpha$ inductively on $\R\backslash C=\bigcup_{n\in\N}B_n$ by $\alpha(\xi)=\nu(\xi)\alpha(\xi)$ for $\xi\in B_n$, noting that $1/2 B_n=B_{n-1}$.  Then $\delta_\alpha=\nu$ on $\R\backslash C/2=\bigcup_{n\in\{0\}\cup\N}B_n$.  To see that $\delta_\alpha=\nu$ on $C/2$, we use the fact that $|m_0||\widehat{\phi}|>0$ on $C/2$ and, for each $\xi\in C/2$,
\begin{align*}
\left((\delta_\alpha-\nu)|m_0||\widehat{\phi}|\right)(\xi)&=\frac{1}{\alpha(\xi)}\left((D^{1,\infty}\alpha)|m_0\widehat{\phi}|-\alpha\nu|m_0\widehat{\phi}|\right)(\xi)\\
&=\frac{1}{\alpha(\xi)}\left(\alpha(2\xi)|\widehat{\phi}|(2\xi)-(m_0\widehat{\phi})(2\xi)\right)\\
&=\frac{1}{\alpha(2\xi)}\left(\widehat{\phi}(2\xi)-\widehat{\phi}(2\xi)\right)\\
&=0.
\end{align*}
This shows that $\delta_\alpha=\nu$ on $\R$, so $\alpha\in\M_{\nu}$ and $\alpha\cdot(\phi_{|m_0|},|m_0|)=(\phi,m_0)$.  Uniquess of $\alpha$ satisfying these conditions is obvious, i.e. in order to have $\alpha\bullet\phi_{|m_0|}=\phi$ it must be that $\alpha=\frac{\widehat{\phi}}{|\widehat{\phi}|}$ on $C=C_\phi$, and so on.
\end{remark}
\begin{remark}
We present a side question: if $C\subset \R$ satisfies (\ref{dyadicpartition}) with $C/2\subseteq C$, what, if any, additional constraints on $C$ are needed to conclude that $LP(C)\neq\emptyset$?  Equivalently, can we construct 1-periodic functions $m_0$ on $S=C+\Z$ satisfying (LP1), (LP2), (LP3) and, if so, how much ``latitude'' do we have in such a construction?
\end{remark}

\subsubsection{The equivalence of $\pmra_1$ and $\pmra_2$\label{section223}}
We now provide a proof of Theorem \ref{pmra1=pmra2}.  This is essentially a shorter version of the argument given for \cite[Theorem 3.8]{PSWXII}.
\begin{enumerate}
\item First, suppose that $\psi\in\pmra_1$.  In other words, we assume that $\psi\in\p$ and that there is a pair $(\phi,m_0)\in\s$ and a unimodular, 1-periodic function $\nu$ so that $\widehat{\psi}(2\xi)=e^{2\pi i\xi}\nu(2\xi)\overline{m_0(\xi+1/2)}\widehat{\phi}(\xi):=m_1(\xi)\widehat{\phi}(\xi)$.  Since $m_0$ is bounded and 1-periodic, it has well-defined Fourier coefficients $(a_k)\in\ell^2(\Z)$.  This means that $\widehat{\phi}(2\xi)$ can be written as $\sum_{k\in\Z}a_ke^{2\pi ik\xi}\widehat{\phi}(\xi)$, so that $\frac{1}{2}\phi(x/2)=\sum_{k\in\Z}a_k\phi(x-k)$, and so $D^{-1}\phi\in\langle\phi\rangle$.  By iteration, this gives $D^{-j}\phi\in\langle\phi\rangle$ for all $j\le 0$.  A similar argument shows that $D^{-j}\psi\in\langle\phi\rangle$ for all $j<0$.  But this means that $V_0(\psi)\subset\langle\phi\rangle$.  So $V_0(\psi)\subseteq \langle\phi\rangle$, which says that $\psi\in\pmra_2$.  

The fact that the $\subseteq$ can be replaced by equality in the last sentence above is a bit more delicate.  To that end, we remark that, since $(\phi,m_0)\in \s$, we have that $\widehat{\phi}(2\xi)=m_0(\xi)\widehat{\phi}(\xi)$ and $|m_0(\xi)|^2+|m_1(\xi)|^2=1$ on $S_\phi$.  However, it will be more convenient to extend the filters $m_0$ and $m_1$ to the whole torus.  To do this, define $m_0(\xi+1/2)=\nu(\xi)\sqrt{1-|m_0(\xi)|^2}$ on $S_\phi\backslash(S_\phi+1/2)$ and make an arbitrary choice on the $\T\backslash(S_\phi\cup(S_\phi+1/2))$ in such a way that $m_0$ satisfies the Smith--Barnwell equation on all of $\T$ --- this is possible since $\T\backslash(S_\phi\cup(S_\phi+1/2))$ is invariant under translation by half-integers.  Thus we may assume that $\widehat{\phi}(2\xi)=m_0(\xi)\widehat{\phi}(\xi)$ and $\widehat{\psi}(2\xi)=m_1(\xi)\widehat{\phi}(\xi)$ for almost every $\xi\in\R$.

Now, observe that
\begin{align*}
\widehat{\psi}(2^j\xi)\overline{\widehat{\psi}(2^j(\xi+k))}&=|m_1(2^{j-1}\xi)|^2\widehat{\phi}(2^{j-1}\xi)\overline{\widehat{\phi}(2^{j-1}(\xi+k))}\\
&=(1-|m_0(2^{j-1}\xi)|^2)\widehat{\phi}(2^{j-1}\xi)\overline{\widehat{\phi}(2^{j-1}(\xi+k))}\\
&=\widehat{\phi}(2^{j-1}\xi)\overline{\widehat{\phi}(2^{j-1}(\xi+k))}-\widehat{\phi}(2^{j}\xi)\overline{\widehat{\phi}(2^{j}(\xi+k))}.
\end{align*}
Plugging in $k=0$, and summing over $j\ge0$, one gets
\[
\sum_{j=1}^\infty|\widehat{\psi}(2^j\xi)|^2=|\widehat{\phi}(\xi)|^2.
\]
(Here one must observe that, for almost every $\xi$, $\lim_{n\rightarrow\infty}|\widehat{\phi}(2^n\xi)|=0$, e.g. \cite[Lemma 2.8]{PSWXI}).  

The conclusion of the first paragraph above is that $\textrm{dim}_{V_0(\psi)}(\xi)\le \textrm{dim}_{\langle\phi\rangle}(\xi)\le 1$, where $\textrm{dim}$ denotes the usual dimension function for a shift-invariant space defined in the introduction; in particular, this function is integer-valued almost everywhere.  The conclusion of the last paragraph above, in the language of \cite{PSWXI} and \cite{PSWXII}, is that $D_\psi=p_\phi$, where
\[
D_\psi=\sum_{k\in\Z}\sum_{j\ge 1}|\widehat{\psi}(2^j(\xi+k))|^2.
\]
Thus in particular, $D_\psi$ and $p_\phi$ are supported on the same sets.  Since $D_\psi\le \textrm{dim}_{V_0(\psi)}(\xi)$, one concludes that, $\textrm{dim}_{V_0(\psi)}=\chi_{S_\phi}$.  This implies that $V_0(\psi)=\langle\phi\rangle$.

\item Now suppose that $\psi\in\pmra_2$.  Then $\textrm{dim}_{V_0(\psi)}\le 1$; in particular it takes on only the values 0 and 1 a.e.  Then one can choose a $\widetilde{\phi}$ so that $V_0(\psi)=\langle\widetilde{\phi}\rangle$ and $\textrm{dim}_{V_0(\psi)}=p_{\widetilde{\phi}}=\chi_{S_{\widetilde{\phi}}}$.  Motivated by our discussion in part 1 of this proof, we define $\phi=\sqrt{D_\psi}\bullet\tilde{\phi}=(\sqrt{D_\psi}\widehat{\widetilde{\phi}})^\vee$.  Then one has, ultimately, that $p_\phi=D_\psi p_{\widetilde{\phi}}=D_\psi$ and $V_0(\psi)=\langle\phi\rangle$.

From here it is not too hard to see, as before, that there are 1-periodic functions $m_0$ and $m_1$ defined for all $\xi\in S_\phi$, so that $\widehat{\phi}(2\xi)=m_0(\xi)\widehat{\phi}(\xi)$ and $\widehat{\psi}(2\xi)=m_1(\xi)\widehat{\phi}(\xi)$.  The result then follows by the following three lemmas:
\begin{lemma}\label{keylemmaequivalent}
The above conditions guarantee that for each $\xi\in S_\phi$,
\[
|\widehat{\phi}(\xi)|^2=\sum_{j=1}^\infty|\widehat{\psi}(2^{j}\xi)|^2
\]
and
\[
|m_0(\xi)|^2+|m_1(\xi)|^2=1.
\]
\end{lemma}
\begin{lemma}
Lemma \ref{keylemmaequivalent} and the discussion preceding it imply the following:
\begin{itemize}
\item $|\widehat{\phi}|$ satisfies the dyadic continuity condition at 0.
\item For almost every $\xi\in S_\phi\cap (S_{\phi}+1/2)$, the matrix 
\[
\mathcal{M}(\xi):=\begin{pmatrix}
m_0(\xi) & m_0(\xi+1/2)\\
m_1(\xi) & m_1(\xi+1/2)
\end{pmatrix}
\]
\end{itemize}
is unitary and hence $m_0$ and $m_1$ satisfy the Smith--Barnwell equations on $S_\phi\cap (S_\phi+1/2)$.  Also, we may choose $\mu\in\U$ for which
\[
m_1(\xi)=e^{2\pi i\xi}\mu(2\xi)\overline{m_0(\xi+1/2)}
\]
on $S_\phi\cap (S_\phi+1/2)$.

\end{lemma}
\begin{lemma}
There is a unique extension of $m_0$ and $m_1$ to $S_\phi\cup (S_\phi+1/2)$.
\end{lemma}
Clearly the three lemmas together give us that $\psi\in\pmra_1$.  The proofs of the latter two are fairly routine --- the second lemma comes down to using the alternating summation idea from the previous section in the context of the Calder\'{o}n and $t_q$ conditions.  

\textbf{Sketch of proof of Lemma \ref{keylemmaequivalent}}: We provide only the following sketch. Consider the $\ell^2(\Z)$-valued map $\Phi(\xi)=(\widehat{\phi}(\xi+k))_{k\in\Z}$ and, for each $j=1,2,3,...$ $\Psi_j(\xi)=(\widehat{\psi}(2^j(\xi+k))_{k\in\Z}$.  By our assumptions, $\C\Phi(\xi)=\textrm{span}\{\Phi_{j}(\xi):j=1,2,3,...\}$.  Then for each such $j$, we have a $c_j(\xi)\in\C$ for which $\Phi_j(\xi)=c_j(\xi)\Phi(\xi)$, and so
\[
\|\Phi_j\|^2_{\ell^2(\Z)}=|c_j(\xi)|^2p_\phi(\xi)=|c_j(\xi)|^2D_\psi(\xi).
\]
Thus
\[
|c_j(\xi)|^2=\frac{\|\Phi_j(\xi)\|_{\ell^2(\Z)}}{D_\psi(\xi)}.
\]
Because $D_\psi(\xi)=\sum_{j=1}^\infty\|\Psi_j(\xi)\|^2_{\ell^2(\Z)}$, this gives us that $|\widehat{\psi}(2^j\xi)|^2=|c_j(\xi)|^2|\widehat{\phi}(\xi)|^2$ for each $j$.  Thus
\[
\sum_{j=1}^\infty|\widehat{\psi}(2^j\xi)|^2=\left(\sum_{j=1}^\infty\|\Psi_j(\xi)\|_{\ell^2(\Z)}^2\right)\frac{|\widehat{\phi}(\xi)|^2}{D_\psi(\xi)}=|\widehat{\phi}(\xi)|^2.
\]
Now, if we pick $\xi\in S_\phi$, then for some $k\in\Z$, $\xi+k=\eta\in C_\phi$.  By 1-periodicity of $m_0$ and $m_1$ on $S_\phi$,
\begin{align*}
(|m_0(\xi)|^2+|m_1(\xi)|^2)|\widehat{\phi}(\eta)|^2&=|\widehat{\phi}(2\eta)|^2+|\widehat{\psi}(2\eta)|^2=\sum_{j=2}^\infty |\widehat{\psi}(2^j\eta)|^2+|\widehat{\psi}(2\eta)|^2=\sum_{j=1}^\infty|\widehat{\psi}(2^j\eta)|^2=|\widehat{\phi}(\eta)|^2.
\end{align*}
From this, we get the desired result on $S_\phi$.

\end{enumerate}

\subsubsection{The passage between $\s$ and $\pmra$}
\begin{definition}\label{fpcdef}
Let $C$ be a subset of $\R$ for which $LP(C)\neq\emptyset$ and let $S=C+\Z$ and $\tilde{S}=S\cup(S+1/2)$.  The set $FP(C)$ of low- and high-pass filters associated with $C$ consists of all pairs $(m_0,m_1)$ of 1-periodic functions on $\tilde{S}$ for which $m_0\big|_S\in LP(C)$ and the matrix 
\[
M_{m_0,m_1}(\xi):=\begin{pmatrix}
m_0(\xi) & m_0(\xi+1/2)\\
m_1(\xi) & m_1(\xi+1/2)
\end{pmatrix}
\]
is unitary for all $\xi\in\tilde{S}$.\footnote{The unitarity of the matrix $M_{m_0,m_1}$ is equivalent to the identity $m_0(\xi)\overline{m_1(\xi)}+m_0(\xi+1/2)\overline{m_1(\xi+1/2)}=0$, which is often referred to as the second Smith--Barnwell Equation.}
\end{definition}
\begin{remark}
If we start with $m_0'\in LP(C)$ then $(m_0,m_1)\in FP(C)$ with $m_0'=m_0\big|_S$ if and only if the following two conditions hold:
\begin{enumerate}
\item There is some unimodular function $\mu_0$ on $(S+1/2)\backslash S$ for which
\[
m_0(\xi)=\mu_0(\xi)\sqrt{1-|m_0'(\xi+1/2)|^2}
\]
for each $\xi\in (S+1/2)\backslash S$.  This is dictated by the Smith--Barnwell condition $|m_0(\xi)|^2+|m_0(\xi+1/2)|^2=1$ for each $\xi\in \tilde{S}$ along with the fact that $m_0'=m_0\big|_S$ satisfies the condition on $S\cap(S+1/2)$.
\item There is some 1-periodic unimodular function $\mu_1$ on $\tilde{S}$ for which $m_1(\xi)=\mu_1(2\xi)e^{2\pi i\xi}\overline{m_0(\xi+1/2)}$ for each $\xi\in\tilde{S}$.  This arises by checking that, for each $m_1$ of this form, $M_{m_0,m_1}$ is unitary on $\tilde{S}$ and then $M_{m_0,\tilde{m}_1}$ is also unitary on $\tilde{S}$ if and only if $\tilde{m}_1=(D^{1,\infty}\mu)m_1$ for some $\mu\in\U$.
\end{enumerate}
\end{remark}
\begin{definition}
When $(m_0,m_1)\in FP(C)$ we have by definition that $m_0\big|_S\in LP(C)$; by the remark following the definition of $LP(C)$, we know there is a choice of $\phi$ so that $(\phi,m_0\big|_S)\in \s$.  For this choice, we define $\psi_{(\phi,m_0,m_1)}\in D\langle\phi\rangle$ by
\[
\psi_{(\phi,m_0,m_1)}=\sqrt{2}D^1(m_1\bullet\phi)
\]
or, equivalently,
\begin{equation}\label{waveletdefinition}
\widehat{\psi}_{(\phi,m_0,m_1)}(\xi)=\left\{
	\begin{array}{ll}
		m_1(\xi/2)\widehat{\phi}(\xi/2)  & \mbox{for } \xi\in 2C \\
		0 & \mbox{for } \xi \notin 2C
	\end{array}
\right..
\end{equation}
\end{definition}
\begin{remark}
Note that $|m_1(\xi/2)|=|m_0(\xi/2+1/2)|=\sqrt{1-|m_0(\xi/2)|^2}>0$ if and only if $|m_0(\xi/2)|<1$ which happens if and only if $|\widehat{\phi}(\xi)|<|\widehat{\phi}(\xi/2)|$.  This condition is trivially satisfied when $\xi\in 2C\backslash C$.  Hence $\textrm{supp }\widehat{\psi}_{(\phi,m_0,m_1)}=(2C\backslash C)\cup\{\xi\in C:|\widehat{\phi}(\xi)|<|\widehat{\phi}(\xi/2)|\}$.  Since $|\widehat{\phi}|=\widehat{\phi}_{|m_0\big|_S|}$, the support of $\widehat{\psi}_{(\phi,m_0,m_1)}$ depends only on $m_0'=m_0\big|_S$.
\end{remark}
\begin{proposition}
If $(m_0,m_1)\in FP(C)$, then let $\psi=\psi_{(\phi,m_0,m_1)}$, where this last quantity is as described in the previous definition.  Then $\psi\in\pmra$.
\end{proposition}
\proof  The fact that $M_{m_0,m_1}$ is unitary establishes that
\begin{equation}\label{eqn15}
\sum_{j=1}^\infty|\widehat{\psi}(2^j\xi)|^2=|\widehat{\phi}(\xi)|^2,
\end{equation}
as we saw in Section \ref{section223}.  As shown in \cite{PSWXI}, this equality implies that $\psi$ satisfies the Calder\'{o}n and $t_q$ equations, so $\psi\in \p$,  Also, $V_0(\psi)=\overline{\sum_{j=1}^\infty W_{-j}(\psi)}=\langle\phi\rangle$, so $\psi\in\pmra$.\qed

\begin{proposition}
Suppose that $\psi\in\pmra$.  Then there exist choices of $(\phi,m_0,m_1)$ for which $\psi=\psi_{(\phi,m_0,m_1)}$ and we then say that each such $\phi$ is a scaling function for $\psi$ with $(m_0,m_1)\in FP(C_\phi)$ the low- and high-pass filters for $\psi$ relative to $\phi$.
\end{proposition}
\proof
This follows by the discussion in the previous subsection.
\qed

\begin{remark}
The above discussion raises a host of questions concerning the back and forth passage between $\s$ and $\pmra$.  If we fix $(\phi,m_0')$ and look at the functions $\mu_0$ and $\mu_1$ --- as in the remark following Definition \ref{fpcdef} --- which describe the pairs $(m_0,m_1)\in FP(C_\phi)$ for which $m_0\big|_{S_\phi}=m_0'$, how does $\psi_{(\phi,m_0,m_1)}$ depend on $\mu_0$ and $\mu_1$?  Alternatively, if we fix $m_0'\in LP(C)$ and look at various choices of $\phi$ for which $(\phi,m_0')\in\s$ along with associated pairs $(m_0,m_1)$, what can we say about the family of functions $\psi_{(\phi,m_0,m_1)}$  Finally, if we fix $\psi\in\pmra$ what can we say about the family of scaling functions for $\psi$ and the associated family of filter pairs $(m_0,m_1)$?  We address these questions in the following subsection.
\end{remark}
\subsection{Comparisons with the back and forth passage between $\s$ and $\pmra$}
\subsubsection{How $\psi_{(\phi,m_0,m_1)}$ depends on $\mu_0$ and $\mu_1$}
Throughout this section, we fix $(\phi,m_0')\in\s$ and let $C=C_\phi$, $S=S_\phi$, $\tilde{S}=S_\phi\cup(S_\phi+1/2)$, and $V_0=\langle\phi\rangle$.  We also fix $(m_0,m_1)\in FP(C)$ for which $m_0\big|_S=m_0'$ and let $\psi=\psi_{(\phi,m_0,m_1)}$ and $W_0=W_0(\psi)=\langle\psi\rangle$.  Then $\psi\in V_1\equiv D^1V_0$, $W_0\cap V_0=\{0\}$, $V_1=\overline{W_0+V_0}$, and $V_0=V_0(\psi)=\overline{\sum_{j=1}^\infty W_{-j}(\psi)}$ --- we emphasize that the sums in the last two equalities are not necessarily orthogonal sums!  When $(\tilde{\phi},\tilde{m}_0')\in\s$ is a scaling function for $\psi$ or some member of $\U\bullet\psi$ (\ref{eqn15}) tells us that $|\widehat{\tilde{\phi}}|=|\widehat{\phi}|$ and hence $|\tilde{m}_0'|=|m_0\big|_S|$.  We also have $\tilde{\phi}\in\langle\phi\rangle$, and it follows from the elementary properties of $\s$ that there is some $\mu\in\U$ for which 
\begin{equation}\label{eqn16}
(\tilde{\phi},\tilde{m}_0')=\mu\cdot(\phi,m_0')=(\mu\bullet\phi,\delta_\mu m_0').
\end{equation}  This motivates us to look only at those members of $\s$ defined by (\ref{eqn16}) for some $\mu\in\U$ and go on to study those members of $\pmra$ having the form $\tilde{\psi}=\psi_{(\tilde{\phi},\tilde{m}_0,\tilde{m}_1)}$ where $(\tilde{m}_0,\tilde{m}_1)\in FP(C)$ and $\tilde{m}_0\big|_S=\tilde{m}_0'=\delta_\mu m_0'$.
\begin{theorem}
In the above setting, we have the following:
\begin{enumerate}
\item \[\tilde{m}_0=\sigma m_0\textrm{ for }\sigma\in\U\textrm{ with }\sigma\big|_S=\delta_\mu.\]
\item \[\tilde{m}_1=(D^{1,\infty}\nu)\sigma m_1\textrm{ for some }\nu\in\U.\]
\item \begin{equation}\label{eqn17}\psi_{(\tilde{\phi},\tilde{m}_0,\tilde{m}_1)}=(\nu D^{-1,\infty}(\mu\sigma))\bullet\psi.\end{equation}
\end{enumerate}
Conversely, for each choice of $\mu,\nu\in\U$ and each $\sigma\in\U$ with $\sigma\big|_S=\delta_\mu$, then (\ref{eqn17}) holds for \[(\tilde{\phi},\tilde{m}_0,\tilde{m}_1)=(\mu\bullet\phi,\sigma m_0, (D^{1,\infty}\nu)\sigma m_1).\]
\end{theorem}
\proof
For $\mu\in\s$ and $(\tilde{\phi},m_0')=\mu\cdot(\phi,m_0')$, a 1-periodic function $\tilde{m}_0$ on $\tilde{S}$ satisfies the Smith--Barnwell equation on $\tilde{S}$ and the extension condition $\tilde{m}_0\big|_S=\tilde{m}_0'=\delta_\mu m_0'$ if and only if, for each $\xi\in (S+1/2)\backslash S=\tilde{S}\backslash S$, one has $|\tilde{m}_0(\xi)|=\sqrt{1-|m_0'(\xi+1/2)|^2}=|m_0(\xi)|$.  But then $|\tilde{m}_0|=|m_0|$ so $m_0'=\sigma m_0$ for $\sigma\in\U$ with $\sigma\big|_S=\delta_\mu$.  Next, if $(\tilde{m}_0,\tilde{m}_1)\in FP(C)$, each of $M_{(m_0,m_1)}(\xi),M_{(\sigma m_0,\sigma m_1)}(\xi)=M_{(m_0',\sigma m_1)}(\xi),$ and $M_{(\tilde{m}_0,\tilde{m}_1)}(\xi)$ are unitary for almost every $\xi\in\tilde{S}$ and, from the remark following Definition \ref{fpcdef}, this condition holds if and only if $\tilde{m}_1=(D^{1,\infty}\nu)\sigma m_1$ for some $\nu\in\U$.  WIth $(\tilde{\phi},\tilde{m}_0,\tilde{m}_1)=(\mu\bullet\phi,\sigma m_0,(D^{1,\infty}\nu)\sigma m_1)$ the fact that $\psi=\sqrt{2}D^1(m_1\bullet\phi)$ implies
\begin{align*}
\tilde{\psi}=\psi_{(\tilde{\phi},\tilde{m}_0, \tilde{m_1})}&=\sqrt{2}D^1(\tilde{m}_1\bullet\tilde{\phi})\\
&=\sqrt{2}D\left[((D^{1,\infty}\nu)\sigma m_1)\bullet(\mu\bullet\phi)\right]\\
&=\sqrt{2}D\left[((D^{1,\infty}\nu)\sigma\mu)\bullet(m_1\bullet\phi)\right]\\
&=(\nu D^{-1,\infty}(\sigma\mu))\bullet\psi,
\end{align*}
so (\ref{eqn17}) holds.
\qed
\begin{remark}
Suppose that $\tilde{\psi}$ is given by (\ref{eqn17}) for some $\mu,\nu\in\U$ and $\sigma\in\U$ with $\sigma\big|_S=\delta_\mu$.  Then $D^{-1,\infty}(\sigma\mu)$ belongs to the group $\U_{-1}=D^{-1,\infty}\U$ consisting of unimodular functions on $\R$ which are $2\Z$-periodic and $\U=\U_0$ is a subgroup of $\U_{-1}$.  Since $2\Z$ is the lattice dual of $\frac{1}{2}\Z$, the group $\U_{-1}$ plays the same role for the principal $\frac{1}{2}\Z$-invariant space $V_1=D^1V_0$ as $\U$ for the principal $\Z$-invariant space $V_0$.  Thus two members of $V_1$ have the same modulus if and only if they lie in the same $\U_{-1}$ orbit while to members of $V_0$ have the same modulus if and only if they lie in the same $\U$ orbit.  In particular, $\tilde{\psi}\in\langle\psi\rangle$ if and only if there is an element $\tilde{\nu}\in\U$ for which $\tilde{\nu}=D^{-1,\infty}(\sigma\mu)$ on the support of $\widehat{\psi}$.  Then, because $\mu$ is only determined by $\widehat{\tilde{\phi}}$ on $C$ and $\sigma$ is only determined by $\mu$ and $\tilde{m}_0$ on $S$, there is no loss of generality in assuming that $\sigma\mu=D^{1,\infty}\tilde{\nu}\in\U_1$.  This reduce (\ref{eqn17}) to
\begin{equation}
\tilde{\psi}=(\nu\tilde{\nu})\bullet\psi.
\end{equation}
\end{remark}
\begin{theorem}
Using the above notation,
\begin{enumerate}
\item $\tilde{\phi}\in L^2(\R)$ is a scaling function for a member of $\U\bullet\psi$ if and only if $\tilde{\phi}\in\U\bullet\phi$.
\item We have $\psi_{(\tilde{\phi},\tilde{m}_0,\tilde{m}_1)}=\psi=\psi_{(\phi,m_0,m_1)}$ if and only if we have $\mu,\nu\tilde{\nu}\in\U$ for which $\tilde{\phi}=\mu\bullet\phi$ and with $\sigma=\frac{D^{1,\infty}\tilde{\nu}}{\mu}$, $\tilde{m}_0=\sigma m_0$, and $\tilde{m}_1=(D^{1,\infty}\nu)\sigma m_1=\frac{D^{1,\infty}(\nu\tilde{\nu})}{\mu}m_1$.
\end{enumerate}
\end{theorem}
\proof
Both of these claims are immediate from (\ref{eqn17}) and the remark preceding the statement of the theorem.

\qed

\begin{remark}
Following up on the remarks preceding the above theorem, when $\tilde{\psi}$ is given by (\ref{eqn17}), the space $W_0(\tilde{\psi})=\langle\tilde{\psi}\rangle\subset V_1$ depends only on the $\U$ coset of $D^{-1,\infty}(\sigma\mu)$ in $\U_{-1}$ or, equivalently the $\U_1$ coset of $\sigma\mu$ in $\U$.  But for the reasons mentioned in the remark following the proof of Lemma \ref{lemma1.4} it is likely impossible to reasonably construct a set of coset representations for $\U/\U_1$ or, equivalently, for $\U_{-1}/\U$.  Also, in order to be more explicit about the choices for $\sigma$, we would need an explicit set of coset representations for $\U/\delta(\U)$, but this, too, is likely impossible to construct.  Despite these difficulties, it's highly likely that further exploration of the objects leading to (\ref{eqn17}), aided by a variety of illuminating examples, would lead to some form of ``tightening'' of the preceding two theorems even though ``complete understanding'' of the sets $LP(C)$ and $FP(C)$ and the way in which they expedite the passage between $\s$ and $\pmra$ seems unrealistic.
\end{remark}
\section{Projections of Scaling Functions}
\subsection{Maximal Principal Shift-invariant Spaces and Projections}
We begin by recalling a couple of definitions:
\begin{definition}\label{maximal}
We say that a principal shift-invariant space $\langle \phi\rangle$ is maximal precisely when it is not a proper subset of any other principal shift-invariant space.  In other words $\langle\phi\rangle\subset \langle\psi\rangle$ implies $\langle\phi\rangle= \langle\psi\rangle$.  We will abuse notation somewhat and refer to a function $\phi$ being maximal (respectively, non-maximal) if the principal shift-invariant space generated by $\phi$ is maximal (respectively, non-maximal).
\end{definition}
\begin{definition}
We will denote by $\s^*$ the subset of $\s$ whose elements are maximal.
\end{definition}
We recall the definition of the bracket: for $\phi,\psi\in L^2(\R)$, the bracket of $\phi$ and $\psi$ is the 1-periodic function in $L^1(\T)$ defined by
\[
[\phi,\psi](\xi):=\sum_{k\in\Z}\widehat{\phi}(\xi+k)\overline{\widehat{\psi}(\xi+k)}\textrm{ a.e.}
\]
The special case where $\phi=\psi$ is especially useful, and we remind the reader of the notation $p_\phi:=[\phi,\phi]$.
\begin{theorem}\label{maximalpositive}
The principal shift-invariant space generated by a function $\phi\in L^2(\R)$ is maximal if and only if $p_\phi(\xi)>0$ for almost every $\xi$.
\end{theorem}
\begin{remark}
This is essentially a folklore result, having been observed independently by numerous researchers.  We provide a short proof for the convenience of the reader.
\end{remark}
\proof
($\Rightarrow$) We prove this by contraposition.  As before, let $S_\phi$ denote the support of $p_\phi$ as a subset of $\R$.  Suppose $S_\phi^c\cap [0,1)$ is a set of positive measure.  Then define a function $\psi$ by $\widehat{\psi}(\xi)=\widehat{\phi}(\xi)$ on $S_\phi$ and $\widehat{\psi}(\xi)=1$ if $\xi\in S_\phi^c\cap[0,1)$.  Then $\psi$ is in $L^2(\R)$ with $p_\psi(\xi)>0$ almost everywhere.  It is easy to verify that $\langle\phi\rangle\subset\langle\psi\rangle$ but if we let $m(\xi)$ denote the 1-periodic function $\chi_{S_\phi^c}\in L^\infty(\T)$, then $0\neq m\bullet\psi\in\langle\psi\rangle$ but $m\bullet\phi\equiv 0$, which gives non-maximality.

($\Leftarrow$) Suppose that $\phi\in L^2(\R)$ is such that $p_\phi(\xi)>0$ almost everywhere and $\langle\phi\rangle\subset\langle\psi\rangle$.  Then we have $\phi\in\langle\psi\rangle$ and so there is some $m$ in the weighted space $L^2(\T,p_\phi)$ so that $\phi=m\bullet\psi$.  Hence $p_\phi=|m^2|p_\psi$.  Thus $|m|^2>0$ almost everywhere and $p_\psi>0$ almost everywhere.  The function $\frac{1}{m}$ is then finite almost everywhere and is still 1-periodic.  We can easily check that $\int \frac{1}{|m|^2}p_\phi=\int p_\psi<\infty$ so that $\frac{1}{m}\in L^2(\T,p_\phi)$.  Thus $\frac{1}{m}\bullet\phi\in \langle\phi\rangle.$  But we have that $\psi=\frac{1}{m}\bullet\phi$, and so $\psi\in\langle\phi\rangle$.  This means that $\langle\phi\rangle=\langle\psi\rangle$, which gives maximality.

\qed
\begin{remark}
It is worth mentioning here that by work of Saliani, \cite{saltrans}, and Paluszy\'{n}ski, \cite{paltrans}, one has that the condition $p_\phi>0$ is equivalent to $\mathcal{B}_\phi:=\{T_k\phi:k\in\Z\}$ being a $\ell^2(\Z)$ independent family.  This $\ell^2(\Z)$ independence is a sort of ``Hilbert space basis'' generalization of linear independence: in particular, it means that if $(a_k)\in\ell^2(\Z)$ and $\lim_{K\rightarrow\infty}\sum_{|k|\le K}a_k T_k\phi=0$ then $(a_k)$ is the zero sequence.  The linear independence of $\mathcal{B}_\phi$ follows by an elegant but very simple argument involving the Fourier transform and the fact that trigonometric polynomials have finitely many zeros.  By comparison, the theorem proven in Saliani's paper uses deep machinery related to the Carleson Theorem (Luzin's Conjecture) about almost everywhere convergence of Fourier series of $L^2(\T)$ functions; in particular, she uses careful estimates related to work by Kisliakov, \cite{kisliakov}, and Vinogradov, \cite{vinogradov}.
\end{remark}

Returning to the previous discussion, let $\phi$ be an $L^2(\R)$ function and $S_\phi$ the support of $p_\phi$.  If $E$ is any measurable subset of $S_\phi$ whose measure is in $(0,1)$, the map $P_E$ which carries $\phi\mapsto\chi_{E}\bullet\phi$ is an orthogonal projection from $\langle\phi\rangle$ onto $\langle\chi_E\bullet\phi\rangle$.  Since $E$ has positive measure, even if $\phi$ is maximal, the function $\chi_E\bullet\phi$ must necessarily be non-maximal.  One also has that $P_E T_k=T_k P_E$ for every integer $k$.  And obviously, $S_{\chi_E\bullet\phi}=E$ and $C_{\chi_E\bullet\phi}=E\cap C_\phi$.

Conversely, suppose that $0\neq\phi\in L^2(\R)$ is non-maximal with $E=S_\phi$.  Then $\phi=P_E\phi^*$, with $\phi^*\in L^2(\R)$ maximal if and only if $\widehat{\phi^*}\big|_E=\widehat{\phi}\big|_E$ and $p_{\phi^*}>0$ on $\T\backslash E$.  For every choice of such $\phi^*$, one has $C_{\phi^*}=C_\phi\cup C^*$, with $C^*$ a set for which $C^*+\Z=\R\backslash E$.

\subsection{Conditions on $E$ for which $P_E(\s)\subset \s$}
When $\phi^*$ is maximal and there is a 1-periodic function $m_0^*$ on $\T$ for which $(\phi^*,m_0^*)\in\s$ and $\phi=P_E\phi^*$ for some $E\subset \T$ with $0<|E|<1$, we have, for each $\xi\in E$
\begin{equation}
\widehat{\phi}(2\xi)=\chi_E(2\xi)\widehat{\phi^*}(2\xi)=\chi_{\frac{1}{2}E}(\xi)m_0^*(\xi)\widehat{\phi}(\xi)\label{eqn19}.
\end{equation}
When $2\xi\in C_\phi=E\cap C_{\phi^*}$, it follows that $\xi\in C_\phi$ with $m_0^*(\xi)\neq 0$.  Then $\frac{1}{2}C_\phi\subset C_\phi$ and $\frac{1}{2}E\subset \tilde{E}=E\cup (E+\frac{1}{2})$.  Putting $C=C_\phi$ and defining $m_0$ on $E$ by $m_0=m_0^*$ on $C/2+\Z$, $m_0=0$ on $C\backslash(C/2)+\Z$, we have that $\frac{1}{\sqrt{2}}D^{-1}\phi=m_0\bullet\phi$ (i.e. $\phi$ is reductive) with $|m_0|>0$ on $C/2+\Z$.  In the following section, we give necessary and sufficient conditions on $E$, $C=E\cap C_{\phi^*}$ and $m_0^*$ for which $(\phi,m_0)\in\s$.
\begin{theorem}
As above, we let $E$ be a measurable subset of $\T$ for which $0<|E|<1$.  Suppose that $\phi^*$ maximal with $(\phi^*,m_0^*)\in\s$.  Let $\phi=P_E\phi^*=\chi_E\bullet\phi^*$.  We use the notation that $C=C_\phi=C_{\phi^*}\cap E$, the support of $\widehat{\phi}$.  We define $m_0$ to be $m_0^*$ on $C/2+\Z$ and $m_0=0$ on $C\backslash(C/2)+\Z$.  Then $(\phi,m_0)\in\s$ if and only if
\begin{enumerate}
\item $\displaystyle{\R=\frac{C}{2}\cup\left(\bigcup_{n=1}^\infty 2^n(C\backslash C/2)\right)}$
\item $|m_0^*|=1$ on $(C/2+\Z)\cap\left((C\backslash C/2)+\frac{1}{2}+\Z\right)$
\item $\displaystyle{\left(C\backslash C/2+\Z\right)\cap\left((C\backslash C/2)+\frac{1}{2}+\Z\right)=\emptyset}$.
\end{enumerate}
\end{theorem}
\proof
From the discussion preceding the statement of this theorem, it is automatic that $(\phi,m_0)$ satisfies $(\s2)$.  We will show that $(\phi,m_0)$ satisfies $(\s 1)$ if and only if property 1 above holds and that $(\phi,m_0)$ satisfies $(\s3)$ if and only if the second and third property hold.

First, note that, since $C/2\subset C$, property 1 is equivalent to the statement that, for each $\xi\notin C$, there exists a positive integer $n_\xi$ such that $2^{-j}\xi\in C/2$ for $j\ge n_\xi$.  Since we have the dyadic continuity property $(\s1)$ for $|\widehat{\phi^*}|$, it follows that property 1 is equivalent to $(\s1)$ for $|\widehat{\phi}|$.

Next, we know that $m_0^*$ satisfies the Smith--Barnwell equation at each $\xi\in \T$, so it's automatic that $|m_0|\le 1$ on $E$ and $|m_0(\xi)|^2+|m_0(\xi+1/2)|^2=1$ for $\xi\in S_0:=(C/2+\Z)\cap(C/2+1/2+\Z)$.  We also have $E\cap(E+1/2)=S_0\cup S_1\cup (S_1+1/2)\cup S_2$, where $S_1=(C/2+\Z)\cap(C\backslash C/2+1/2+\Z)$ and $S_2=(C\backslash C/2+\Z)\cap(C\backslash C/2+1/2+\Z)$.  Since $m_0=0$ on $(C\backslash C/2)+\Z)$, properties 2 and 3 are necessary and sufficient for $m_0$ to satisfy the Smith--Barnwell equations for $E\cap (E+1/2)$, i.e. for $(\phi,m_0)$ to satisfy $(\s3)$.

\qed

\subsection{Every Element of $\s$ is the Projection of an Element of $\s^*$}
The following is our main result and demonstrates the importance of the maximal scaling functions.

\begin{theorem}
Suppose that $\phi\in\s$ is non-maximal.  Then there is a $\phi^*\in\s^*$ so that $\phi=\chi_{S_\phi}\bullet\phi^*$.
\end{theorem}
\begin{remark}
For a given $\phi$, it is highly unlikely that the choice of $\phi^*$ satisfying the conclusion of the above theorem will be unique.
\end{remark}
\proof
As usual, let $C=C_\phi$, $S=S_\phi=C+\Z$ and let $m_0$ be the 1-periodic function on $S$ for which $\frac{1}{\sqrt{2}}D^{-1}\phi=m_0\bullet\phi$.  Consider the function $m_0$ on the set $C/2+\Z$: by the aforementioned definition of $m_0$, one has that $m_0$ is nonzero (a.e.) on $C/2+\Z$.  We shall define an extension, $m_0^*$ of $m_0\big|_{C/2+\Z}$.  We note that $S/2$ is the disjoint union of $C/2+\Z$ and $C/2+1/2+\Z$ --- this follows by an elementary calculation.  On $C/2+1/2+\Z$, define $m_0^*$ via 
\[
m_0^*(\xi)=\nu(\xi)\sqrt{1-|m_0(\xi+1/2)|^2},
\]
for some measurable, 1-periodic, unimodular function $\nu$; this is well-defined since $|m_0|\le 1$ on $S$.  Now, we observe that since $S/2$ is invariant under translation by $\Z/2$, we have that $(S/2)^c$ is also invariant under such translations.  Thus we may find a measurable set $B$, which is invariant under translation by $\Z$, so that $(S/2)^c$ is the disjoint union of $B$ and $B+1/2$.  Now one merely needs to define a unit vector $(m_0^*(\xi),m_0^*(\xi+1/2))$ on $B$ (in a measurable way) so that neither component is 0.  Clearly this can be arranged in a multitude of ways.  The $m_0^*$ constructed in this way then satisfies the Smith--Barnwell equation on all of $\R$.

Because $(\phi,m_0)$ satisfy $(\s1)$ and $m_0^*=m_0$ on $S/2$, it follows that
\[
\lim_{n\rightarrow\infty}\prod_{j=n}^\infty |m_0^*(2^{-j}\xi)|=1\textrm{ for each }\xi\in\R,
\]
and, as we have previously argued, $(\phi_{|m_0^*|},|m_0^*|)\in\s$ with $\widehat{\phi}_{|m_0^*|}(\xi)=\prod_{j=1}^\infty |m_0^*(2^{-j}\xi)|$ for each $\xi$.  Then $|\widehat{\phi}|=\widehat{\phi}_{|m_0^*|}$ on $S$, so $|\widehat{\phi}|=\chi_S\widehat{\phi}_{|m_0^*|}$.  We can choose $\mu\in\M$ for $m_0^*=\mu|m_0^*|$.  Then as we discussed in \ref{glpf}, we may construct $\alpha\in\M$ for which $\delta_\alpha=\mu$ and, on $C$, we have $\widehat{\phi}=\alpha|\widehat{\phi}|$.  This means that $\alpha\cdot(\phi_{|m_0^*|,|m_0^*|})=(\phi^*,m_0^*)\in\s$, with $\phi=\chi_S\bullet\phi^*$.  To summarize, we have that
\begin{equation}\label{temp248}
\widehat{\phi^*}(\xi)=\alpha(\xi)\prod_{j=1}^\infty|m_0^*(2^{-j}\xi)|,
\end{equation}
where $\alpha$ is a unimodular function so that $\frac{\alpha(2\xi)}{\alpha(\xi)}$ is 1-periodic.  

It remains to prove that $\phi^*$ is in $\s^*$, i.e. that $\phi^*$ is maximal in the sense of Definition \ref{maximal}.  To that end, let $Z^*$ denote the zero set of $m_0^*$.  By construction, we observe that $Z^*\subset C+1/2+\Z$.  The zero set of $\widehat{\phi^*}$ is then $\bigcup_{n=1}^\infty 2^nZ^*:=(C^*)^c$.  The claim is that $C^*+\Z=\R$, up to a set of measure zero.  Suppose not.  Let $A=\bigcap_{n\in\Z}(C^*)^c+n$.  If $A$ has measure zero, then it is not hard to see that $C^*+\Z=\R$ up to a set of measure zero.  So suppose to the contrary that $A$ has positive measure.  We first make the observation that $A+1=A$.  Now, let $\xi\in A$.  Then $\xi_1=2^m\xi$ for some $\xi\in Z^*$ and positive integer $m$, as $A\subset (C^*)^c$.  Since $\xi\in C+1/2+\Z$, we know that $\xi+1/2\notin C+1/2+\Z$ (otherwise $|m_0(\xi)|^2+|m_0(\xi+1/2)|^2=0$, in violation of Smith--Barnwell).  Thus there is some integer $\beta$ so that $\widehat{\phi^*}(\xi+1/2+\beta)\neq 0$.

Now, since $A+1=A$, we have that $\xi_1+2^{m-1}(2\beta+1)\in A$.  Thus $\widehat{\phi^*}(\xi_1+2^{m-1}(2\beta+1))=0$.  Then by iteration of the two-scale equation,
\begin{align*}
0&=\widehat{\phi^*}(\xi_1+2^{m-1}(2\beta+1))\\
&=\widehat{\phi^*}(2^m\xi+2^{m-1}(2\beta+1))\\
&=m_0^*(2^{m-1}\xi+2^{m-2}(2\beta+1))\widehat{\phi^*}(2^{m-1}\xi+2^{m-2}(2\beta+1))\\
&=...\\
&=m_0^*(2^{m-1}\xi+2^{m-2}(2\beta+1))\cdot\cdot\cdot m_0^*(\xi+\beta+1/2))\widehat{\phi^*}(\xi+\beta+1/2).
\end{align*}

As stated at the beginning of the previous paragraph, the last term in the product of the previous line is nonzero.  Likewise, by 1-periodicity and Smith--Barnwell, $|m_0(\xi+\beta+1/2)|=|m_0(\xi+1/2)|=1$, since $m_0(\xi)=0$.  Thus the last two terms are nonzero.  Note that if $m=1$, this immediately produces a contradiction, so we may assume that $m>1$.  In such a case, there is some integer $k$ with $1\le k\le m-1$ for which $m_0(2^{m-k}\xi+2^{m-k-1}(2\beta+1))=0$.  By one periodicity, this means that $m_0^*(2^{m-k}\xi)=0$, meaning that $2^{m-k}\xi\in Z^*$.  Since $Z^*$ is 1-periodic, it follows that $2^{-k}\xi_1=2^{m-k}\xi\in A$.  Iterating this argument, we produce a strictly increasing sequence $k_n$ so that $\widehat{\phi^*}(2^{-k_j}\xi_1)=0$ for all $n$.  But this violates the dyadic continuity property of $\widehat{\phi^*}$, which contradicts the fact that $\phi^*\in\s$.  Hence $\phi^*$ indeed must be in $\s^*$, concluding the proof.

\qed

This theorem can be interpreted as a $\pmra$ cousin of a theorem of Naimark.  We will explain this (see the remark below) through a corollary of the above theorem.  Before stating the corollary, we give a definition.

\begin{definition}
Suppose that $\psi\in\p$.  Letting $W_j=D^j\langle\psi\rangle$, we say that $\psi$ is a semiorthogonal Parseval frame wavelet if $W_{j_1}\perp W_{j_2}$ whenever $j_1\neq j_2$.  We let $\pso$ denote the class of all such $\psi$.
\end{definition}

\begin{corollary}
Suppose that $\psi\in\pmra\cap\pso$, with $(\phi,m_0,m_1)$ the associated scaling function, low-pass filter, and high-pass filter.  Then there is a scaling function $\phi^*$ for an orthonormal MRA wavelet so that $\phi=\chi_{S_\phi}\bullet\phi^*$.
\end{corollary}
\proof
Semiorthogonality of $\psi$ guarantees that $D_\psi(\xi)=\textrm{dim}_{V_0(\psi)}(\xi)$ which must be the characteristic function of a measurable set (this is \cite[Corollary 3.2]{PSWXII} translated into current terminology).  But it also holds that $p_\phi=D_\psi$, as argued in the proof of Theorem \ref{pmra1=pmra2}.  Thus $p_\phi$ is the characteristic function of a set.

Now, we have by the theorem we just proved that there is a $\psi^*\in\pmra$ with associated scaling function $\phi^*$ so that $\phi=\chi_{S_\phi}\bullet\phi^*$, with $p_{\phi^*}>0$ and $D_{\psi^*}=p_{\phi^*}$.  Now define $\theta$ to be\footnote{We should mention that this choice of $\theta$ is related to the notion of semiorthogonalization; see \cite{SSW}}
\[
\widehat{\theta}(\xi)=\frac{1}{\sqrt{D_{\psi^*}}}\widehat{\phi^*}(\xi).
\]
This method of choosing $\theta$ It is then easy to check that $\langle\theta\rangle=\langle\phi^*\rangle$, that $p_\theta\equiv 1$, and that $\phi=\chi_{S_\phi}\bullet\theta$.  Since $p_\theta\equiv 1$, it follows that $\{\theta(\cdot-k):k\in\Z\}$ forms an orthonormal basis for $\langle\theta\rangle$.  Since $\langle\theta\rangle=\langle\phi^*\rangle$, we have that $\theta$ is the scaling function for an orthonormal MRA, and thus is the scaling function for an orthonormal MRA wavelet $\psi_\theta$ as in \cite{HWWave}.

\qed

\begin{remark}
The above is, in some sense, an analogue of Naimark's theorem that if $\{v_n:n\in\Z\}$ is a Parseval frame on a Hilbert space $\mathcal{H}$, then $\mathcal{H}$ can be linearly and isometrically embedded into a larger Hilbert space $\mathcal{K}$ so that the image of $\{v_n:n\in\Z\}$ in $\mathcal{K}$ is an orthogonal projection of an orthonormal basis for $\mathcal{K}$.  The previous corollary then says that a similar statement holds for $\pmra\cap\pso$ at the level of the scaling functions.
\end{remark}

\bibliography{projections}   
\bibliographystyle{plain}   

\end{document}